\documentclass[11pt]{article}

\usepackage{float}
\usepackage{subfigure}
\usepackage{subfiles}
\usepackage{amsmath, amsthm, amsfonts}
\usepackage{mathtools}
\usepackage{appendix}
\usepackage[square, numbers]{natbib}
\usepackage{graphicx}
\usepackage{caption}
\usepackage{todonotes}
\usepackage{xcolor}
\usepackage{multirow}

\usepackage[normalem]{ulem} 

\renewcommand{\u}{\boldsymbol{u}}

\newcommand{\n}{\boldsymbol{n}}

\newcommand{\f}{\boldsymbol{f}}
\newcommand{\g}{\boldsymbol{g}}
\newcommand{\h}{\boldsymbol{h}}
\newcommand{\w}{\boldsymbol{w}}

\renewcommand{\v}{\boldsymbol{v}}

\newcommand{\T}{\mathcal{T}}
\newcommand{\C}{\mathcal{C}}
\newcommand{\B}{\mathcal{B}}
\newcommand{\N}{\mathcal{N}}
\newcommand{\R}{\mathcal{R}}

\newcommand{\bsff}{\boldsymbol{\mathsf{f}}}

\newcommand{\bsfg}{\boldsymbol{\mathsf{g}}}

\newcommand{\bw}{\boldsymbol{w}}

\newcommand{\btau}{\boldsymbol{\tau}}
\newcommand{\bsigma}{\boldsymbol{\sigma}}

\newcommand{\pdiff}[2]{\frac{\partial #1}{\partial #2}}

\newtheorem{theorem}{Theorem}[section]

\theoremstyle{definition}

\newtheorem{remark}[theorem]{Remark}

\usepackage{soul} 

\graphicspath{{sep_figs}}

\title{Entropy-stable in- and outflow boundary conditions for the compressible Navier-Stokes equations}
\author{Magnus Sv{\"a}rd$^1$ and Anita Gjesteland$^2$\\ \footnotesize    $^1$ Department of Mathematics, University of Bergen, P.O Box 7803, N-5020 Bergen, Norway \\ \footnotesize $^2$ Department of Applied Mathematics, University of Waterloo, Waterloo, ON, Canada, N2L 3G1 \\
\small $^1$ magnus.svard@uib.no, $^2$ agjesteland@uwaterloo.ca
}

\begin{document}

\maketitle

\begin{abstract}
We propose inflow and outflow boundary conditions for the compressible Navier-Stokes equations and prove that they allow a priori estimates of the entropy, mass and total energy. Furthermore, we demonstrate how to approximate these boundary conditions in conjunction with an entropy-stable finite-volume scheme. The method is also applicable to other types of entropy-stable schemes. Finally, we carry out some numerical computations with the finite-volume scheme and demonstrate their robustness.
\end{abstract}

\section{Introduction}

When computing fluid flows with the compressible Euler or Navier-Stokes equations, the boundary conditions are quintessential in defining the model. They model the effects that objects submerged in the fluid have on the flow; they define what flows in and out of the domain, or they mimic an open boundary infinitely far away from the domain of interest.

Mathematically, the PDE (partial differential equation) and its boundary conditions must form a well-posed problem, i.e., a unique and stable solution to the problem must exist in a well-defined function space. Knowing in what function space the solution resides, is necessary when designing convergent numerical schemes.

Here, we are focusing on the compressible Navier-Stokes equations. Although its well-posedness remains elusive, the linearised well-posedness and associated numerical stability with various boundary conditions have been extensively studied. See e.g. \cite{Nordstrom95(1),HesthavenGottlieb96,NordstromSvard05,SvardCarpenter07}. By using summation-by-parts schemes (\cite{SvardNordstrom14,FernandezHicken14}), it is now well-known how to design linearly stable and possibly high-order accurate numerical approximation schemes. 

The importance of linear stability when solving the non-linear Navier-Stokes equations was established in \cite{Strang64}, where it was shown that linearly stable schemes converge as long as the solution to the non-linear PDE is smooth. However, solutions are not always smooth and there is a need to extend the theory to allow for possibly non-smooth solutions. Although, we do not currently have a well-posedness theory, it will undoubtedly require the initial-boundary value problem and its numerical discretisation scheme to satisfy a number of non-linear a priori estimates. 

Here, we consider three fundamental estimates that are available for the compressible Navier-Stokes equations. Namely, conservation of mass, conservation of total energy and boundedness of the entropy (the second law of thermodynamics). The first two are readily obtained for conservative discretisation schemes, and entropy bounds can be derived for entropy-stable schemes (\cite{Tadmor03,Carpenter_etal16}). 

To obtain entropy bounds for initial-boundary value problems, entropy-stable
numerical approximation schemes need to be closed with appropriate boundary approximations that enforce the boundary conditions.  For  the Euler equations subject to wall boundaries, entropy stable numerical fluxes were proposed in \cite{SvardOzcan14}. Generalisations of these boundary conditions, as well as strong impositions of the no-slip boundary condition for the Navier-Stokes equations were later derived in e.g. \cite{SvardCarpenter18,ParsaniCarpenter15,ChanLin22,GjestelandSvard22}.

Other types of boundary conditions have proven much harder to derive entropy bounds for. For the Euler equations, attempts to design entropy-stable boundary conditions (other than walls) are found in \cite{OlssonOliger94,SvardOzcan14,SvardMishra12,Svard21}. For open problems governed by the compressible Navier-Stokes equations, i.e., for objects submerged in a free stream that extends to infinity, entropy-stable boundary conditions, and associated entropy-stable approximation schemes, were derived in  \cite{SvardGjesteland24}. However, these boundary conditions cannot be used on a finite domain, which precludes the modelling of inlets and outlets of a channel.

In \cite{Svard25}, the problem was finally solved for the Euler equations. Boundary conditions, and numerical boundary fluxes, were derived that can be used at any type of inlets and outlets, including channels and open boundaries. They were proven to be linearly stable and non-linear bounds for the entropy, mass and total energy were derived. Moreover, they are fully compatible with previously derived entropy-stable wall boundary conditions.

In this paper, we generalise the boundary conditions derived in \cite{Svard25} for the Euler equations to the compressible Navier-Stokes equations.
\begin{itemize}
\item For the continuous equation, we derive non-linear a priori estimates for mass, total energy and entropy.
\item We propose numerical boundary fluxes that enforce these conditions and prove the corresponding semi-discrete a priori bounds.
  \item The current procedure is compatible with previously derived entropy-stable wall boundary conditions. 
\end{itemize}
We emphasise that the only a priori assumptions we need on the (numerical) solution is that the thermodynamic variables remain positive. We do not need to a priori assume that density, velocity or temperature remain bounded, in order to derive the non-linear bounds.


\section{The compressible Navier-Stokes equations}

\subsection{The initial-boundary value problem}

To reduce notation, and without restriction, we consider the two-dimensional case (2-D) with only inflow or outflow boundaries. We will briefly highlight the necessary modifications to handle the one-dimensional (1-D) and three-dimensional (3-D) cases. 

Thus, in 2-D, let $\Omega$ be a domain with boundary $\partial \Omega$. We define $\n=(n^1,n^2)^\top$ as the outward pointing unit normal and  $\btau$ the (counter clock-wise) unit tangent vector to $\partial \Omega$. 

The 2-D compressible Navier-Stokes equations govern the evolution of the conserved variables,   
$$\u = \begin{bmatrix}\rho, & m, & n, & E\end{bmatrix}^\top,$$
denoting density, momentum (in the $x$- and $y$-directions), and total energy, respectively. On $\Omega\times (0,\T]$, the Navier-Stokes equations take the form,
\begin{align}\label{eq:NS}
	\pdiff{\u}{t} + \pdiff{\f^I}{x} + \pdiff{\g^I}{y} & = \pdiff{\f^V}{x} + \pdiff{\g^V}{y}.
\end{align}
The inviscid fluxes are given by
\begin{equation}\label{eq:inviscid fluxes}
	\begin{aligned}
		\f^I  = \left(\begin{matrix} \rho u \\ \rho u^2 + p \\  \rho uv \\ u(E+p)\end{matrix}\right), \hspace{2em}
		\g^I  = \left(\begin{matrix} \rho v \\ \rho uv \\ \rho v^2 + p \\ v(E+p) \end{matrix}\right),
	\end{aligned}
\end{equation}
where $u = \frac{m}{\rho}$ and $v = \frac{n}{\rho}$ are the velocity components in the $x$- and $y$-directions, respectively; they form the velocity vector $\v^\top=(u,v)$. The pressure is given by $p = (\gamma-1) (E- \frac{1}{2} \rho(u^2 + v^2))$ where the constant $\gamma= \frac{c_p}{c_v}$ is the ratio between the specific heat capacities at constant pressure and volume. Moreover, the temperature is given by the ideal gas law, $T = \frac{p}{\rho R}$, where $R = c_p - c_v$ is the gas constant. The sound speed is $c=\sqrt{\gamma RT}$ and the specific (physical) entropy is $S =c_v \ln(p/\rho^\gamma)$. Finally, we define the local Mach number as $Ma=\frac{|\v|}{c}$. 

The viscous fluxes are
\begin{equation}\label{eq:viscous_fluxes}
	\begin{aligned}
		\f^V  = \left(\begin{matrix}  0\\ \tau_{xx}\\  \tau_{xy} \\ u\tau_{xx}+v\tau_{xy}+\kappa T_x\end{matrix}\right), \hspace{2em}
		\g^V  = \left(\begin{matrix}  0 \\ \tau_{xy} \\ \tau_{yy}\\ u\tau_{xy}+v\tau_{yy}+\kappa T_y \end{matrix}\right),
	\end{aligned}
\end{equation}
where $\kappa$ is the heat conductivity. Here, we use Stokes' hypothesis, which leads to
\begin{align*}
  \tau_{xx} &= \frac{4}{3} \mu u_x -\frac{2}{3}\mu v_y,\\
  \tau_{xy} &= \mu (u_y+v_x),\\
  \tau_{yy} &= \frac{4}{3}\mu v_y -\frac{2}{3}\mu u_x,
\end{align*}
where $\mu$ is the dynamic viscosity.
Solutions to (\ref{eq:NS}) are required to be \emph{admissible}.  i.e., only solutions with $p,\rho,T>0$ are allowed. Therefore, we take initial data to be uniformly bounded in $H^1(\Omega)$ and with $p,\rho,T$ bounded away from zero.

To define the boundary conditions, we need boundary data (denoted with a subscript $b$) at $(x,y,t)\in \partial\Omega\times [0,\T]$. In particular,   $v_{b\tau}(x,y,t)$ denotes boundary data for the tangential component of the velocity at the boundary and $\v_b(x,y,t)=(u_b,v_b)$ is boundary data for the full velocity vector. The boundary data are subject to the same restrictions as the initial data. That is, $\rho_b(x,y,t)$, $p_b(x,y,t)$, $v_{b\tau}(x,y,t)$, $\v_b(x,y,t)$ are  bounded in $H^1(\partial\Omega\times(0,\T])$ and $\rho_b,p_b, T_b\geq \epsilon>0$. For consistency, if both $\rho_b$ and $p_b$ are used at the boundary they define, through the gas law, the boundary temperature $T_b$ and the speed of sound $c_b$.
\begin{remark}
Although we provide data for the solution variables, they are not boundary conditions. E.g., we do not enforce $p=p_b$ on any boundary. The boundary conditions will be given below and they use the available boundary data. 
\end{remark}

Next, we split the boundary $\partial \Omega$ into four subsets:
\begin{align*}
  \partial\Omega^{sup.in}&=\{(x,y)\in\partial\Omega :\v\cdot\n<0,\,\,|\v\cdot\n|>c_b\},\\
  \partial\Omega^{sub.in}&=\{(x,y)\in\partial\Omega:\v\cdot\n<0,\,\,|\v\cdot\n|\leq c_b\},\\
  \partial\Omega^{sub.out}&=\{(x,y)\in\partial\Omega:\v\cdot\n>0,\,\,|\v\cdot\n|\leq c\},\\
  \partial\Omega^{sup.out}&=\{(x,y)\in\partial\Omega:\v\cdot\n>0,\,\,|\v\cdot\n|>c\}.
\end{align*}

Furthermore, we define
\begin{align*}
\partial \Omega^{in}&=\partial\Omega^{sup.in}\cup \partial\Omega^{sub.in},\\
\partial \Omega^{out}&=\partial\Omega^{sup.out}\cup \partial\Omega^{sub.out}.
\end{align*}
 For physical consistency at supersonic-inflow points ($\partial \Omega^{sup.in}$), we also require that $\v_b\cdot \n<0$ and $|\v_b\cdot \n|>c_b=\sqrt{\gamma p_b/\rho_b}$.


We propose the following boundary conditions at inflows and outflows:
\begin{align}
 n^1(\f^I-\f^V)+n^2(\g^I -\g^V) = [n^1\f_b+n^2 \g_b](x,y,t)\label{bcs}
\end{align}
where $[n^1\f_b+n^2 \g_b]$ is a \emph{boundary data flux vector}. To define the boundary data flux, we need a velocity vector $(u^*,v^*)$ that coincides with $v_n=\v\cdot\n$ and whose tangential component is given by the boundary data $v_{b\tau}$. 
These requirements lead to
\begin{align*}
  u^*=v_nn^1-v_{b\tau}n^2,\\
  v^*=v_nn^2+v_{b\tau}n^1.
\end{align*}
(See \cite{Svard25} for a detailed explanation.)

The boundary data flux vectors  for the different boundary types are the same as in the Euler case  (see \cite{Svard25} Eqn. (20)-(23)), which ensures that \eqref{bcs} reduces to the Euler boundary conditions in the inviscid limit.

$\textbf{Supersonic inflow}\quad (\partial \Omega^{sup.in}):$
\begin{equation}\label{b_flux_supin}
  \begin{aligned}
    [n^1\f_b+n^2 \g_b]&=
    \left(\rho_b \v_b\cdot \n,\rho_bu_b \v_b\cdot \n +n^1p_b,\rho_b v_b \v_b\cdot \n + n^2p_b, \v_b\cdot \n(E_b+p_b) \right)^\top,\\
E_b&=\frac{1}{2}\rho_b(u_b^2+v_b^2)+\frac{p_b}{\gamma-1}.
\end{aligned}
\end{equation}

$   \textbf{Subsonic inflow}\quad ( \partial \Omega^{sub.in}):$
\begin{equation}\label{b_flux_subin}
  \begin{aligned}
   [n^1\f_b+n^2 \g_b]&=
    \left(\rho_bv_n,\rho_b v_n u^*+n^1p_b,\rho_b v_n v^*+n^2p_b,v_n(E^*+p_b) \right)^\top \\
    E^*&=\frac{1}{2}\rho_b(v_n^2+v_{b\tau}^2)+\frac{p_b}{\gamma-1}
\end{aligned}
\end{equation}

$    \textbf{Subsonic outflow}\quad (\partial \Omega^{sub.out}):$
\begin{equation}\label{b_flux_subout}
  \begin{aligned}
    [n^1\f_b+n^2 \g_b]&=
    \left(\rho v_n,\rho u v_n +n^1p_b,\rho v v_n + n^2p_b, v_n(E+p) \right)^\top
\end{aligned}
\end{equation}
Note that pressure data is only used in the momentum flux components.

$    \textbf{Supersonic outflow}\quad (\partial \Omega^{sup.out}):$
\begin{equation}\label{b_flux_supout}
  \begin{aligned}
         [n^1\f_b+n^2 \g_b]&= n^1\f^I+n^2 \g^I,
\end{aligned}
\end{equation}
which is equivalent to setting the normal viscous flux to zero.
\begin{equation}\label{b_flux_supout_2}
  \begin{aligned}
    - n^1\f^V-n^2\g^V&=0 \quad \textrm{on}\quad \partial \Omega^{sup.out}.
\end{aligned}
\end{equation}
Note that in all cases but the last, it is the total flux (not the inviscid flux) that is (partially) given by data. 


\subsection{Compatibility of data}\label{sec:comp}

We have already mentioned that the data for velocity and speed of sound at supersonic inflows, must be compatible with a supersonic inflow. 

Furthermore, in order to divide the boundary into subdomains with sub- or supersonic flows, the reference speed of sound must be unambiguous. That requires that $p_b(x,y,t)$ and $\rho_b(x,y,t)$ at a point $(x,y)\in \partial\Omega$ are independent of the current local flow conditions. That is, we \emph{can not} set
\begin{align*}
\rho_b(x,y,t)&=f_1(x,y,t),\quad p_b(x,y,t)=g_1(x,y,t),\quad x\in{ \partial} \Omega^{sub.in},\\
\rho_b(x,y,t)&=f_2(x,y,t),\quad p_b(x,y,t)=g_2(x,y,t),\quad x\in{ \partial} \Omega^{sup.in},
\end{align*}
 when the flow shifts between sub- or supersonic at a boundary point. We have to use $f_1=f_2$ and $g_1=g_2$ such that $\rho_b,p_b$ are both time and space continuous. 

For similar reasons, $p_b(x,y,t)$ on the boundary must not shift its value discontinuously when the flow shifts between subsonic inflow and subsonic outflow. 


Hence, to allow for continuous solutions, we must require that the data functions are continuous in time and space (along the boundary). Moreover, the initial data must be compatible with the boundary data (see \cite{GKO}).

\subsection{1-D and 3-D modifications}

In 1-D, the tangential velocity component is zero and the  velocity $u$ is also the normal component. Hence, the subsonic inflow data flux is
\begin{equation}\label{b_flux_subin_1-D}
  \begin{aligned}
    \f_{b}&=
    \left(\rho_bu,\rho_bu^2+p_b,u(E^*+p_b) \right)^\top\\ 
    E^*&=\left(\frac{1}{2}\rho_bu^2+\frac{p_b}{\gamma-1}\right)
\end{aligned}
\end{equation}
The modifications to obtain the other three flux data vectors are analogous. 

In 3-D, there are two fluxes $\h^I$ and $\h^V$ in the $z$-direction, and the fluxes and solution are extended to five components. The velocity and normal vectors also have a $z$-component, $w$ and $n^3$. At the boundary there is also a second tangential component $\bsigma$, orthogonal to both $\btau$ and $\n$ (forming a right-handed coordinate system). The equations take the form
\begin{align*}
	\pdiff{\u}{t} + \pdiff{\f^I}{x} + \pdiff{\g^I}{y}+\pdiff{\h^I}{z} & = \pdiff{\f^V}{x} + \pdiff{\g^V}{y}+\pdiff{\h^V}{z},
\end{align*}
where 
\begin{align}
 n^1(\f^I-\f^V)+n^2(\g^I -\g^V)+n^3(\h^I-\h^V) &=\nonumber\\
 [n^1\f_b+n^2 \g_b+n^3\h_b](x,y,t)\nonumber
\end{align}
In particular, the subsonic inflow conditions become,
\begin{equation*}
\begin{aligned}
     n^1\f_b+n^2 \g_b +n^3\h_b=\\ 
 \left(\rho_bv_n,v_n\rho_bu^*+n^1p_b,v_n\rho_b v^*+n^2p_b,v_n\rho_b w^*+n^3p_b,v_n(E^*+p_b) \right)^\top,
\end{aligned}
\end{equation*}
where
\begin{equation*}
    E^*=\left(\frac{1}{2}\rho_b(v_n^2+v_{b\tau}^2+v_{b\sigma}^2)+\frac{p_b}{\gamma-1}\right),
\end{equation*}
 and $v_{b\sigma}$ is the given bounded boundary-data function for the second tangential velocity component. The other boundary conditions are generalised to 3-D following the same pattern.


\subsection{The number of boundary conditions}

The number of boundary conditions required for well-posedness of the linearised Navier-Stokes equations was determined in \cite{GustafssonSundstrom78} and is given in Table \ref{tab1}.
\begin{table}[htb]
    \centering 
    \caption{Number of boundary conditions for the Navier-Stokes equations.}
    \begin{tabular}{|l | c | c | c |}
\hline
      flow regime & 3-D & 2-D & 1-D \\      \hline
        inflow & 5  & 4 & 3 \\ \hline
         outflow &  4 & 3 &  2\\ \hline
    \end{tabular}\label{tab1}
\end{table}

In the proposed boundary conditions (\ref{b_flux_supin})-(\ref{b_flux_supout}), the viscous flux appears in the boundary condition, which implies that all but the first component are always specified. In the continuity equation the viscous flux is zero and no data is given at outflows. Hence, there is no boundary condition for the continuity equation at outflows. At inflows, boundary data is provided for the first component of the inviscid flux, i.e., all components are specified at inflows. Hence, the  boundary conditions (\ref{b_flux_supin})-(\ref{b_flux_supout}) are consistent with Table \ref{tab1}. 

\subsection{Alternative formulation of the boundary conditions}\label{sec:alt_bc}

To give some further insights into the subsonic boundary conditions, we recast them for a boundary where $x=constant$. In that case, $n^1=1$, $n^2=0$, $v_n=u^*=u$ and $v^*=v_{b\tau} \eqqcolon v_b$, where $v_b$ is the $y$-component of the velocity and the subscript $b$ signifies that, in this case, it is given by data.

First, we consider the subsonic outflow conditions. Inserting \eqref{b_flux_subout} into \eqref{bcs} for an $x=constant$ boundary, leads to
\begin{align*}
\rho u &= \rho u,\\
\rho u^2 +p -f^V_2 &= \rho u^2 + p_b,\\
\rho v u -f^V_3&=\rho v u, \\
u(E+p)-f^V_4&= u(E+p).
\end{align*}
where $f^V_i$, $i=2,3,4$ denote the non-zero components of $\f^V$. The first line does not constitute a boundary condition. The next three simplify to,
\begin{align*}
p -f^V_2 &=  p_b,\\
f^V_3&= 0,  \\
f^V_4&= 0.
\end{align*}
These three boundary conditions are exactly the same as the pressure boundary conditions proven linearly stable in \cite{SvardCarpenter07} (sec. 3.3 and appendix B).

In the same way, the subsonic inflow conditions, in their original form, become
\begin{align*}
\rho u &= \rho_b u,\\
\rho u^2 +p -f^V_2 &= \rho_b u^2 + p_b,\\
\rho uv -f^V_3&=\rho_b u v_b, \\
u(E+p)-f^V_4&= u\left(\frac{1}{2}\rho_b(u^2+v_b^2) + \frac{\gamma p_b}{\gamma-1}\right),
\end{align*}
From the first of these conditions, we deduce that $\rho=\rho_b$. Using this condition in the other three, we obtain the following four equivalent conditions,
\begin{align*}
\rho  &= \rho_b, \\
p -f^V_2 &=  p_b,\\
\rho_b uv -f^V_3&=\rho_b u v_b, \\
u(E+p)-f^V_4&= u\left(\frac{1}{2}\rho_b(u^2+v_b^2) + \frac{\gamma p_b}{\gamma-1}\right).
\end{align*}

\subsection{Non-linear a priori bounds}\label{sec:apriori}

To set the scene for the entropy analysis, we begin by listing a number of well-known definitions and relations:
\begin{align*}
 S&=c_v\ln(p/\rho^\gamma)&  \textrm{specific entropy},\\
 U(\u)& = -\rho S, &  \textrm{strictly convex entropy function},\\
  F(\u)&=-\rho u S,&  \textrm{entropy flux in x-direction},\\
  G(\u)&=-\rho v S,&  \textrm{entropy flux in y-direction},\\
    \psi^x(\u)&= U_{\u} \f^I-F=c_v(\gamma-1)\rho u, & \textrm{entropy flux potential in x-direction},\\
    \psi^y(\u)&= U_{\u} \g^I-G=c_v(\gamma-1)\rho v, &   \textrm{entropy flux potential in y-direction}.
\end{align*}
We also need the entropy variables
    \begin{align*}
 \w^\top&=U_{\u} 
 = \left( -S + c_v\gamma - \frac{u^2 + v^2}{2T},   \frac{u}{T},  \frac{v}{T},  -\frac{1}{T}\right), 
\end{align*}
and that the entropy fluxes by definition satisfy,
\begin{align*}
     U_{\u} \f^I_{\u} = F_{\u}, \quad
  U_{\u} \g^I_{\u} = G_{\u}. 
\end{align*}


The  main result in this paper is given in the following theorem.
\begin{theorem}\label{theo1}
  Admissible solutions to (\ref{eq:NS}) with boundary conditions (\ref{bcs})-(\ref{b_flux_supout}), satisfy the following \emph{a priori} estimates
  \begin{align*}
  \{  \rho,\,
    \rho|\v|^2,\,
    p,\,
    \rho\log(\rho),\,
    \rho\log(T)\}&\in L^{\infty}(0,\T;L^1(\Omega)).
  \end{align*}
  
\end{theorem}
\begin{proof}
  The proof is a straightforward generalisation of that in \cite{Svard25}. We begin with  the continuity equation,
\begin{align*}
\partial_t  \int_\Omega \rho \, d\Omega = -\int_{\partial \Omega } (\rho \v)\cdot \n\,ds.
\end{align*}
Noting that $\rho>0$, the integral on the left-hand side is $\|\rho\|_{L^1(\Omega)}$ and we obtain
\begin{align}
\partial_t  \|\rho\|_{L^1(\Omega)} + \int_{\partial \Omega^{out} } (\rho \v)\cdot \n\,ds= -\int_{\partial \Omega^{in} } (\rho \v)\cdot \n\,ds.\label{rho_est}
\end{align}
The boundary integral on the left-hand side is positive since $\v\cdot \n>0$ on outflows. Furthermore, at supersonic inflows the integrand on the right-hand side is fully given by data. At subsonic inflows, the normal velocity is bounded by the sound speed, which in turn is given by data since both density and pressure are. Integration in time yields the desired $L^\infty(0,\T;L^1(\Omega))$ bound on $\rho$. Furthermore, we obtain a bound on the boundary integral,
\begin{equation}\label{b_est_1}
\begin{aligned}
\int_0^\T  \int_{\partial \Omega^{out} }\rho\v\cdot \n\,ds dt\leq \C.
\end{aligned}
\end{equation}

Similarly,  the energy equation yields
\begin{align*}
\partial_t  \|E\|_{L^1{(\Omega)}}+\int_{\partial \Omega^{out} }(E+p)\v\cdot \n-(\f^{V}_4,\g^{V}_4)\cdot \n \,ds &=\\
-\int_{\partial \Omega^{in} }(E+p)\v\cdot \n-(\f^{V}_4,\g^{V}_4)\cdot \n\,\,ds.
\end{align*}
At both subsonic and supersonic outflow boundaries, the boundary condition is equivalent to setting the normal component of the viscous flux to zero in the energy equation. This leaves a positive boundary integrand on the left. Using the inflow conditions, we obtain
\begin{align}
  \partial_t  \|E\|_{L^1{(\Omega)}}+\int_{\partial \Omega^{out} }(E+p)\v\cdot \n\,ds &=\nonumber \\
  -\int_{\partial \Omega^{sup.in} }(E_b+p_b)\v_b\cdot \n\,ds-\int_{\partial \Omega^{sub.in} }(E^*+p_b)\v\cdot \n\,ds.
  \label{E_est}
\end{align}
Since at subsonic inflows $\v\cdot \n$ is bounded by $c_b$, we see that both integrals on the right are bounded. Upon integration in time, we conclude that $\|\rho|\v|^2\|_{L^1(\Omega)}$ and $\|p\|_{L^1(\Omega)}$ are bounded for all $t\in[0,\T]$, and also that
\begin{equation}\label{b_est_2}
\begin{aligned}
\int_0^\T  \int_{\partial \Omega^{out} }(E+p)\v\cdot \n\,ds dt\leq \C.
\end{aligned}
\end{equation}

Next, we turn to the entropy,
\begin{align*}
  \int_\Omega \w^\top \pdiff{\u}{t} \: d\Omega + \int_\Omega \w^\top \left( \pdiff{\f^I}{x} + \pdiff{\g^I}{y} \right) \: d\Omega & =  \int_\Omega \w^\top \left( \pdiff{\f^V}{x} + \pdiff{\g^V}{y} \right) \: d\Omega. &
\end{align*}
We use that $\w^\top \pdiff{\f^I}{x} = F_x=(\w^\top\f-\psi^x)_x$ and obtain
\begin{align*}
	\int_\Omega \pdiff{U}{t} \: d\Omega +\int_{\Omega}\w_x^\top \f^V+\w_y^\top\g^Vd\Omega&=\\
    - \int_{\partial \Omega} \w^\top\left(\left(\f^I-\f^V,\g^I-\g^V\right) \cdot \n \right)-(\psi^x,\psi^y)\cdot\n\: ds.
\end{align*}
Insert the boundary condition,
\begin{align*}
\f_b=\left(\f^I-\f^V,\g^I-\g^V\right) \cdot \n,
\end{align*}
to obtain
\begin{align}
	\int_\Omega \pdiff{U}{t} \: d\Omega +\int_{\Omega}\w_x^T\f^V+\w_y^T\g^Vd\Omega =- \int_{\partial \Omega} \w^\top\f_b-(\Psi^x,\Psi^y)\cdot\n\: ds\label{ent_est}
\end{align}
Since the second term on the left-hand side is the non-negative entropy dissipation (see e.g. \cite{SvardCarpenter18}, Eqn (8)-(14)), a bound is obtained if the right-hand side is controlled. We note that $\f_b$ is exactly the same as in \cite{Svard25} (see Eqn. (21)-(23)). Hence, the right-most integrand is exactly the same as (30) in \cite{Svard25} where, with the help of (\ref{b_est_1}) and (\ref{b_est_2}), it was shown to be bounded. We omit the remaining details. From the bounds (\ref{rho_est}), (\ref{E_est}), (\ref{ent_est}) and positivity, the bounds on $\rho \log\rho$ and $\rho \log T$ follow in the same way as in \cite{Svard22}.


\end{proof}

\subsection{ Linear well-posedness}
We have currently not succeeded in proving that the proposed boundary conditions are also linearly well-posed. Specifically, it is the subsonic inflow case that is challenging. 

First, the supersonic inflow and supersonic outflow are the same as 
in \cite{SvardCarpenter07} and therefore linearly well-posed. Second, as discussed in Section \ref{sec:alt_bc}, the subsonic outflow is the same as the pressure outflow condition proven to be linearly well-posed in \cite{SvardCarpenter07}.

As for the subsonic inflow, we observe that since inflows require 5 boundary conditions in three space dimensions, the supersonic inflow condition can be used in the subsonic inflow case as well, resulting in both linear well-posedness and the non-linear bounds. However, in this alternative the inflow velocity is specified, which does not appear to allow smooth transitions between inflow and outflow; the transition from outflow to inflow occurs when the normal velocity traverses zero, and then it is abruptly set to a non-zero value in the boundary flux. Furthermore, the use of the supersonic inflow condition at subsonic inflows will not lead to the Euler conditions (\cite{Svard25}) as $\mu\rightarrow 0$. For these reasons, we have not analysed this option further. 


\section{Finite-volume approximation scheme}

To mimic the continuous estimates and obtain the density and energy bounds, we use a numerical method approximating the divergence form of the equation. Moreover, it must satisfy the summation-by-parts property (\cite{SvardNordstrom14}) to produce the correct boundary terms in the estimates. Finally, the scheme must be entropy stable at all interior points. 

\begin{remark}
  In \cite{Svard25} it was remarked that the divergence requirement can be relaxed if the sole goal is to obtain the a priori estimates. However, to approximate physically correct weak solutions, the divergence form is necessary (\cite{HouLeFloch94,LaxWendroff60}).
  Local entropy stability can also be relaxed to the global counterpart at the expense of possible local violations of the second law of thermodynamics.

  Finally, we do not consider entropy-variable approximations of the viscous terms. Compared with approximations in the conservative variables, such approximations require much stronger a priori estimates to infer weak convergence and we do not consider them a viable path towards non-linearly convergent schemes.
\end{remark}

Since we propose a new boundary treatment for in- and outflows, there is no need for an excessively sophisticated interior scheme; that would only complicate notation. The simplest possible scheme meeting our requirements is a node-centred finite-volume scheme. Since second-derivatives are challenging to define in a consistent way on unstructured grids \cite{Eymard_etal97,SvardGong07,SvardNordstrom04}, we use a Cartesian one. 

\subsection{Preliminaries}

For the domain $\Omega=[0,L_x]\times [0,L_y]$, we introduce a set of grid points, or equivalently nodes, $(x_k,y_l)$, $x_k=kh_x$, $y_l=lh_y$, $k,l=0...N,M$, where $h_x=L_x/N$ and $h_y=L_y/M$ are the grid spacings.  The dual volume, $\Omega_{i,j}$, for grid point $i,j$ is 
\begin{align*}
  (x,y)\in \Omega_{i,j}\quad \textrm{if}\quad\quad&  ih_x-h_x/2< x\leq ih_x+h_x/2,\\
&  jh_y-h_y/2< y\leq jh_y+h_y/2,\\
&  1\leq i\leq N-1,\quad 1\leq j\leq M-1
\end{align*}
At a boundary point, $\Omega_{i,j}$ is delimited by $\partial \Omega$. Thus,
\begin{align*}
  (x,y)\in \Omega_{0,j}, \quad \textrm{if} \quad \quad &0\leq x\leq h_x/2,\\
  &  jh_y-h_y/2< y\leq jh_y+h_y/2,\\
  & 1\leq j\leq M-1
\end{align*}
and similarly at the other three boundaries. In the same way, we have
\begin{align*}
(x,y)\in \Omega_{0,0}, \quad \textrm{if} \quad 0\leq x\leq h_x/2,\ 0\leq y\leq h_y/2,
\end{align*}
at the bottom-left corner and similarly in the other three corners. The measure of $\Omega_{i,j}$ is denoted $V_{i,j}$.

To be able to write the scheme on standard finite-volume form, we map the Cartesian indices to a single consecutive index,
\begin{align}
i=kN+lM, 
\end{align}
such that $i\in \{0,1,2,...,NM\}$. All double-index entities are then recast to a single index using this mapping, e.g, $\Omega_{k,l}\to \Omega_i$. This allows us to use two indices for quantities associated with an edge between two points $i$ and $j$. In a Cartesian mesh, the dual volumes are all rectangles. Circling around the dual-volume boundary counter-clockwise, we define the signed coordinate differences $\Delta y_{ij}$ and  $\Delta x_{ij}$ for the dual-volume side associated the edge between $i$ and its neighbour $j$. Since the grid is Cartesian, the left and right $\Delta y_{ij}$ will be non-zero and $\Delta x_{ij}$ are zero; at the top and bottom it is the other way around.
\begin{remark}
We use a comma $k,l$ when referring to the Cartesian indexing and no comma, $ij$, for entities associated with an edge between grid point $i$ and $j$.
\end{remark}
The variables are associated with a node $i$, which is denoted with a subscript $i$. For instance, the density and speed of sound at node $i$ are denoted as $\rho_i$ and $c_i$. By extending the point value to a constant in $\Omega_{i}$, we obtain a piecewise constant numerical solution on $\Omega$.

At boundary points, the boundary of the dual volume coincides with $\partial \Omega$. The corresponding signed differences for the dual-volume boundary are denoted $\Delta x_{ii}$, $\Delta y_{ii}$. Note that this works at corners as well where $\Delta x_{00}=h_x/2$ for the lower part (where $\Delta y$ is zero), and $\Delta y_{00}=-h_y/2$ for the left part (where $\Delta x$ is zero); the boundary flux data vector will be different for the two contributions since they represent different boundaries. If $(x_i,y_i)$ is not a boundary point, these differences are identically zero.

The unit normal associated with the external boundary at $(x_i,y_i)$ is defined as,
\begin{align*}
 \n_{ii}= (n^1,n^2)_{ii}&=(\Delta y_{ii},-\Delta x_{ii})/dl_{ii},\quad\quad
  dl_{ii}= \sqrt{\Delta y_{ii}^2+\Delta x_{ii}^2}
\end{align*}
Furthermore, the normal velocity at a boundary point is given by
\begin{align*}
(v_n)_{ii}=u_{i} n^1_{ii}+v_{i}n^2_{ii},
\end{align*}
and $\btau_{ii}=(-n^2_{ii},n^1_{ii})$ is the unit tangential vector, such that
the tangential velocity is
\begin{align*}
  (v_\tau)_{ii}=-u_{i} n_{ii}^2+v_{i}n_{ii}^1.
\end{align*}

 \begin{figure}[ht]
\centering
\includegraphics[width=\textwidth]{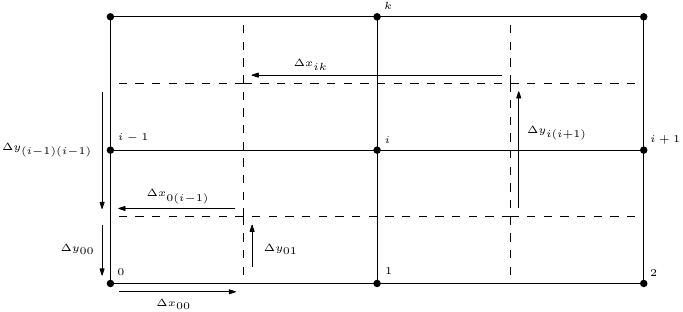}
\caption{Depiction of a grid. The solid lines mark the primary grid and the dashed lines the dual grid.}\label{fig:grid}
\end{figure}

$N_i$ is the set of indices of nodes that are connected to $(x_i,y_i)$ by an edge. If $i$ is a boundary node, then the index $i$ is also included in $N_i$ and represents the contribution from the physical boundary. Furthermore,  the set of all points is denoted $\N$, and the set of boundary points is denoted as $\B$.

For  a boundary point $i\in\B$, we characterise the flow as
\begin{align*}
  (v_n)_{i}<0, \quad & |(v_n)_{i}|\geq c_b(x_i,y_i,t) &\Rightarrow&\quad \textrm{supersonic inflow},\\
  (v_n)_{i}<0, \quad &|(v_n)_{i}|< c_b(x_i,y_i,t) &\Rightarrow& \quad\textrm{subsonic inflow},\\
  (v_n)_{i}>0, \quad &|(v_n)_{i}|< c_{i} &\Rightarrow&\quad \textrm{subsonic outflow,}\\
  (v_n)_{i}>0, \quad & |(v_n)_{i}|\geq c_{i} &\Rightarrow& \quad\textrm{supersonic outflow}.\\
\end{align*}
At time $t$, we define $\B^{out}$ to be the set of all outflow points and   $\B^{in}$ the set of all inflow points. These sets are further split in $\B^{sup.in}$, $\B^{sub.in}$, $\B^{sub.out}$ and $\B^{sub.out}$ according to the above rules.

\begin{remark}
  Since the sides of the Cartesian grid may be subject to different boundary data, one may keep track of each side of the corner cell and impose the correct data. (See also discussion in Sec. \ref{sec:comp}.) However, seeing the corner cell as entirely belonging to one or the other boundary will not formally degrade the accuracy of the  scheme. This is inevitable in some cases, e.g. if the left boundary is a wall where the no-slip is imposed strongly, and the bottom is an outlet. Then the point with index $i=0$ (low-left corner) is either considered to be a wall point, in which case the boundary fluxes are unnecessary when updating the velocity since it is set strongly. If it is not considered to be the wall, then the entire corner cell is treated like an inflow. Finally, note that these remarks are, by and large, independent of the order of accuracy of the scheme. To achieve high-order accuracy, the solution must be highly regular, which requires data to be compatible to a high degree of approximation, which is unlikely to be the case at corners between inflows and walls. 
\end{remark}

We will use the viscous approximation scheme proposed in \cite{SvardCarpenter18}, and also used in \cite{GjestelandSvard22}. To define this scheme, we need some further notation. We also remark that we only need flux approximations at the interior points. At the boundary the total flux is given by the data flux vector. Thus, using the Cartesian indexing, we have
\begin{align*}
  D^xa_k&=D^xa_{i,j}=\frac{a_{i+1,j}-a_{i-1,j}}{2h_x},\quad 1\leq i\leq N-1,\, 0\leq j\leq M,
\end{align*}
and similarly in the $y$-direction.

Furthermore, we define the mean values,
\begin{align*}
  \overset{x}{a}_k=\overset{x}{a}_{i,j}=\frac{a_{i+1,j}+a_{i-1,j}}{2},\quad 1\leq i\leq N-1,\, 0\leq j\leq M,\\
  \overset{y}{a}_k=\overset{y}{a}_{i,j}=\frac{a_{i,j+1}+a_{i,j-1}}{2},\quad 0\leq i\leq N,\, 1\leq j\leq M-1.\\
\end{align*}
The following approximations of "1" will also be used,
\begin{align*}
(1_x)_k=(1_x)_{i,j}=\frac{T^{-1}_{i,j}}{\overset{x}T^{-1}_{i,j}}, \quad\quad
(1_y)_k=(1_y)_{i,j}=\frac{T^{-1}_{i,j}}{\overset{y}T^{-1}_{i,j}}.
\end{align*}
 These are necessary to use in the numerical viscous fluxes given below, to allow for the viscous terms to form a positive definite term in the numerical entropy estimate. (See \cite{SvardCarpenter18} for further details.)


\subsection{ The finite-volume scheme}

Turning to the numerical approximation of \eqref{eq:NS}, we use a standard node-centred scheme (see \cite{NordstromForsberg03}),
\begin{align}
V_{i}(\u_{i})_t + \left[ \sum_{j \in N_i} (\f_{ij}^I-\f_{ij}^V) \Delta y_{ij} \right]- \left[ \sum_{j \in N_i} (\g^I_{ij}-\g_{ij}^V) \Delta x_{ij}  \right], & \quad i \in\N, \label{scheme}
\end{align}
where $\u$ is the full solution vector, and $\u_{i}$ a 4-component subvector holding the approximate solution at  $(x_i,y_i,t)$. Furthermore, $\bsff^{I,V}_{ij}$ and $\bsfg^{I,V}_{ij}$ are consistent approximations of the inviscid and viscous fluxes. We require that the inviscid fluxes between two nodes ($i\neq j$) are entropy-stable approximations satisfying
\begin{equation}\label{shuffle}
	\begin{aligned}
 \tfrac{\bw^\top_j - \bw^\top_i}{2} \f^I_{ij} \Delta y_{ij} - \tfrac{\psi^x_{j} - \psi^x_{i}}{2} \Delta y_{ij}  \leq 0,  \\
		 - \tfrac{\bw^\top_j - \bw^\top_i}{2} \g^I_{ij} \Delta x_{ij} + \tfrac{\psi^y_{j} - \psi^y_{i}}{2} \Delta x_{ij}  \leq 0.
	\end{aligned}
\end{equation}
where $\bw_i=\bw(\u_i)$ are the entropy variables.  (See  \cite{Tadmor03,IsmailRoe09,Chandrashekar13,Svard22} for examples of such fluxes.)

The approximations of the viscous fluxes were developed in \cite{SvardCarpenter18} for a Cartesian finite-difference scheme,
\begin{align*}
  \f^V_{i}=(0,1_x\tau_{xx},1_x\tau_{xy},1_x(\overset{x}{u} \tau_{xx}+\overset{x}{v}\tau_{xy}) + \kappa D_xT)^T_{i},\\
  \g^V_{i}=(0,1_y\tau_{xy},1_y\tau_{yy},1_y(\overset{y}{u} \tau_{xy}+\overset{y}{v}\tau_{yy}) +\kappa D_y T)^T_{i},
\end{align*}
where, at the point $i$,
\begin{align*}
  \tau_{xx}&=\mu (\frac{4}{3}D^xu-\frac{2}{3}D^yv), \\
  \tau_{xy}&=\mu(D^yu+D^xv), \\
  \tau_{yy}&=\mu (\frac{4}{3}D^yv-\frac{2}{3}D^xu). \\
\end{align*}
As in the continuous equations, we have invoked Stokes' hypothesis in the formulae above. The viscous fluxes along an edge are subsequently computed as the mean values,
\begin{align*}
  \f_{ij}^V=\frac{\f^V_{i}+\f^V_{j}}{2},\\
  \g_{ij}^V=\frac{\g^V_{i}+\g^V_{j}}{2}.\\
\end{align*}
Analogous to the continuous case, we define at subsonic inflow boundary points a velocity where the tangential component is given by data,
\begin{align*}
  (u^*)_i=(v_n)_in^1_{ii}-v_{b\tau}(x_i,y_i,t)  n^2_{ii},\\
  (v^*)_i=(v_n)_in^2_{ii}+v_{b\tau}(x_i,y_i,t)n^1_{ii}.
\end{align*}

At the boundaries, we set,
\begin{align}
(n^1(\f^I-\f^V)+n^2(\g^I-\g^V))_{ii}=(n^1\f_b+n^2\g_b)_{ii}.\label{num_bc}
\end{align}
The discrete boundary data fluxes $(n^1\f_{b}+n^2\g_{b})_{ii}$  are constructed in the same way as the continuous ones. 
\begin{remark}
Note that there is no need to define an approximation of the normal boundary 
flux (left-hand side of \eqref{num_bc}), since we simply insert the right-hand side of \eqref{num_bc} in the scheme at the boundary points. 
\end{remark}

For the time and space dependent boundary data, we use the notation $(\rho_b)_i=\rho_b(x_i,y_i,t)$, $i\in \B^{in}$, and analogously for the other boundary-data functions. We also use the superscripts $\rho$ and $E$ to denote the first and last component of the flux-approximation vector.

$\textbf{Supersonic inflow:}$
\begin{equation}\label{b_flux_supin_disc}
  \begin{aligned} 
(n^1\f_b+n^2\g_b)_{ii}=&\phantom{+}n^1_{ii}\left(\rho_bu_b,\rho_bu_b^2+p_b,\rho u_bv_b,u_b(E_b+p_b) \right)^\top\Big|_{(x_i,y_i,t)}\\
&+n^2_{ii}\left(\rho_bv_b,\rho_bu_bv_b,\rho v_b^2+p_b,v_b(E_b+p_b) \right)^\top\Big|_{(x_i,y_i,t)},\\
E_b(x_i,y_i,t)=&\,\frac{1}{2}\rho_b(u_b^2+v_b^2)+\frac{p_b}{\gamma-1}\Big|_{(x_i,y_i,t)}.
\end{aligned}
\end{equation}

$   \textbf{Subsonic inflow:}$
\begin{equation}\label{b_flux_subin_disc}
  \begin{aligned}
 (n^1\f_b+n^2\g_b)_{ii}   &=
    \left(\rho_bv_n,v_n\rho_bu^*+n^1p_b,v_n\rho_b v^*+n^2p_b,v_n(E^*+p_b) \right)^\top\Big|_{(x_i,y_i,t)},\\
    E^*_{ii}&=\frac{1}{2}\rho_b(v_n^2+v_{b\tau}^2)+\frac{p_b}{\gamma-1}\Big|_{(x_i,y_i,t)}.
\end{aligned}
\end{equation}

 $   \textbf{Subsonic outflow:}$
\begin{equation}\label{b_flux_subout_disc}
  \begin{aligned}
(\f_{b})_{ii}&=\left(\rho u,\rho u^2+p_b,\rho uv,u(E+p) \right)^\top\Big|_{(x_i,y_i,t)},\\
(\g_{b})_{ii}&=\left(\rho v,\rho uv,\rho v^2+p_b,v(E+p) \right)^\top\Big|_{(x_i,y_i,t)}.
\end{aligned}
\end{equation}

$    \textbf{Supersonic outflow:}$
\begin{equation}\label{b_flux_supout_disc}
  \begin{aligned}
(\f_{b})_{ii}&=\left(\rho u_,\rho u^2+p,\rho uv,u(E+p) \right)^\top\Big|_{(x_i,y_i,t)},\\
(\g_{b})_{ii}&=\left(\rho v,\rho uv,\rho v^2+p,v(E+p) \right)^\top\Big|_{(x_i,y_i,t)}.
\end{aligned}
\end{equation}
Note that the variables at the boundary points are updated by the scheme; providing data for the boundary fluxes constitutes a weak imposition of the boundary conditions. 
\begin{remark}
These boundary conditions reduce to the Euler conditions in \cite{Svard25}, if $\mu=\kappa=0$. 
\end{remark}

\subsection{A priori estimates}


\subsubsection{Mass balance}

The mass balance is obtained by summing the continuity equation over all grid points. 
\begin{align}
\sum_{i= 1}^\N V_i(\rho_i)_t +\sum_{i\in\B^{out}} \left(\f^\rho_{b,ii}\Delta y_{ii}-\g^\rho_{b,ii}\Delta x_{ii}\right)=\sum_{i\in\B^{in}} \left(-\f^\rho_{b,ii}\Delta y_{ii}+\g^\rho_{b,ii}\Delta x_{ii}\right).\label{rho_balance}
\end{align}
The right-hand side is a sum over the inflow points where $(v_n)_{ii}\leq 0$.
We split the right-hand side into two sums, one for $\B^{sup.in}$ and one for $\B^{sub.in}$. Beginning with the subsonic inflow points, we have $\f^\rho_{b,ii}=(\rho_b)_iu_i$ and  $\g^\rho_{b,ii}=(\rho_b)_i(v_n)_{ii}$ leading to
\begin{align*}
 -\f^\rho_{b,ii}\Delta y_{ii}+\g^\rho_{b,ii}\Delta x_{ii} = -(\rho_b)_i(v_n)_{ii}dl_{ii},\quad i\in \B^{sub.in}.
\end{align*}
At these points, the right-hand side is positive, but $|(v_n)|_{ii}\leq c_{b,i}$, $(i\in\B^{sub.in})$, leading to the desired bound. In the supersonic case $(i\in\B^{sup.in})$, $(v_n)_{ii}$ is also specified  by data, immediately leading to a bound.

At outflow points ($i\in \B^{out}$), $(v_n)_{ii}\geq 0$ (and $\rho_i>0$) implying that the outflow boundary sum is non-negative. Hence, we obtain the following a priori bound,
\begin{align}
\int_0^T\sum_{i\in\B^{out}} \left(\f_{b,ii}^\rho\Delta y_{ii}-\g_{b,ii}^\rho\Delta x_{ii}\right)\,dt\leq \C,\label{m_boundary}
\end{align}
along with an $L^1$ bound on density for all $t$.

\subsubsection{Energy balance}
Turning to the energy balance, we have
\begin{align*}
\sum_{i=1}^\N V_i(E_i)_t+\sum_{i\in\B^{out}} \left(\f^E_{b,ii}\Delta y_{ii}-\g^E_{b,ii}\Delta x_{ii}\right) &=\\
\sum_{i\in\B^{in}} \left(-\f^E_{b,ii}\Delta y_{ii}+\g^E_{b,ii}\Delta x_{ii}\right).
\end{align*}
At supersonic inflow points, the fluxes are completely defined by the data vector, such that the left-hand side is bounded for all $i\in \B^{sup.in}$. At the subsonic inflow ($i\in \B^{sub.in}$), the normal velocity is not given by data but satisfy $|(v_n)|_{ii}\leq c_{b,i}$, and we deduce from (\ref{b_flux_subin_disc}) that
\begin{align*}
 -\f^E_{b,ii}\Delta y_{ii}+\g^E_{b,ii}\Delta x_{ii} = -(v_n)_{ii}dl_{ii}\left(\frac{1}{2}\rho_{i}((v_n)_{ii}^2+(v_{b\tau})_{ii}^2)+ \frac{\gamma p_b}{\gamma-1}\right)\leq \C dl_{ii}.
\end{align*}
At both types of outflow, no boundary data is used to construct the flux. However, the terms are all positive implying that,
\begin{align}
\int_0^T\sum_{i\in\B^{out}} \left(\f_{b,ii}^E\Delta y_{ii}-\g_{b,ii}^E\Delta x_{ii}\right)\,dt\leq \C,\label{E_boundary}
\end{align}
and we also collect an $L^1$ bound on $E$ for all $t$. 

\subsubsection{Entropy balance}

The numerical approximations of entropy potentials are $\psi^{x,y}_i=\psi^{x,y}(\u_i)$  at all grid points, $i\in \N$, and the entropy flux approximations associated with the edges are
\begin{equation}\label{ent_flux}
  \begin{aligned}
    F_{ij} &= \tfrac{\bw^\top_j + \bw^\top_i}{2}\f^I_{ij}-\frac{\psi^x_i+\psi^x_j}{2},\\
    G_{ij}&= \tfrac{\bw^\top_j + \bw^\top_i}{2}\g^I_{ij}-\frac{\psi^y_i+\psi^y_j}{2}.
  \end{aligned}
\end{equation}
The entropy balance is obtained by  multiplying \eqref{scheme} by $\bw_i^\top$ and summing over all points. Using \eqref{ent_flux} at all interior edges, the shuffle condition \eqref{shuffle}, and the numerical boundary flux (\ref{num_bc}) , we arrive at
\begin{align}
\sum_{i= N}^NV_iU(\u_i)_t  +\sum_{i\in \B}\left(\left(\bw^\top_i\f_{b,ii}-\psi^x_{i}\right)\Delta y_{ii}-\left(\bw^\top_i\g_{b,ii}-\psi^y_{i}\right)\Delta x_{ii}\right) \leq DIFF\label{ent_balance}.
\end{align}
where
\begin{align*}
DIFF = \sum_{i\in \N}\sum_{j\in \N_i\cap i} \Delta \bw^\top_{ij}(\f^V_{ij}\Delta y_{ij}-\g^V_{ij}\Delta x_{ij}).
\end{align*}
and $\Delta \bw_{ij}=\bw_j-\bw_i$.
Since the grid is Cartesian and the viscous approximations are exactly the same as in \cite{SvardCarpenter18}, it follows (from that paper) that $DIFF\leq 0$. 

Hence, stability follows if the boundary terms,
\begin{align}
  (n^1)_{ii}(\bw_i^\top\f_{b,ii}-\psi^x_i)+  (n^2)_{ii}(\bw_i^\top\g_{b,ii}-\psi^y_i)\geq \C,\label{disc_bt}
\end{align}
are bounded from below, which was shown in \cite{Svard25}. We conclude that $\sum_{i= 1}^\N V_iU(\u_i)_t\leq \C$.

From positivity and the bounds on entropy, energy and  density,  we obtain the following semi-discrete a priori bounds
\begin{align*}
\{ \rho, \rho \log \rho\, \rho |\v|^2, \rho \log T, p\}\in L^{\infty}(0,\T,L^1(\Omega)).
\end{align*}

\subsection{Other entropy-stable schemes}\label{sec:other_schemes}

Here, we have chosen to present the boundary conditions, and numerical boundary conditions using a conservative entropy-stable finite-volume scheme with the standard (conservative-variable) form of the viscous terms.

However, the boundary procedures are largely decoupled from the interior discretisations as long as the latter satisfy a few properties. 
\begin{itemize}
\item The scheme must preserve positivity of $\rho$ and $p$.
\item The scheme must be on divergence (conservation form). This ensures the $L^1$ estimates on $\rho$ and $E$. 
\item The inviscid flux approximations must be entropy stable and produce the same boundary terms in the entropy balance as the current finite volume scheme. The viscous approximations must be entropy dissipative in the interior (i.e., $DIFF\leq 0$) and produce the same boundary terms as the current scheme.  
\end{itemize}
As noted in \cite{Svard25}, the requirement that the scheme is on divergence form  can be omitted as long as the other two are met and one is only concerned with the approximation of \emph{smooth solutions}. However, note that for smooth solutions, linear stability is sufficient for convergence (see \cite{Strang64}).


\subsection{Consistency with no-slip wall}\label{sec:no-slip}

It is also important that inflow and outflow boundary conditions are compatible with other boundary procedures. This is indeed the case, since for schemes satisfying the above list of properties, each boundary point (actually boundary face) contributes with a term to the estimates that is independent of the terms from all other boundary points/faces. Hence, boundary points can be assigned as walls or in/outflows independently.

In particular, the provably entropy-stable strong imposition of the no-slip boundary condition along with a weak imposition of the adiabatic wall condition proposed in \cite{GjestelandSvard22} is seamlessly compatible with the boundary conditions proposed in this paper.

\section{Numerical experiments}


\subsection{Vortex impinging outflow boundary}
To test the boundary conditions, we introduce a vortex in the domain and let it travel towards the outflow boundary. The vortex is a proxy for typical flow structures that are produced by objects in a freestream and subsequently convected downstream. The particular vortex we use is found in \cite{ErlebacherHussaini97}. It is an analytical and isentropic solution to the compressible Euler equations. Here, we use it as initial data and as the computation is running it is diffused by the parabolic terms, i.e., the flow is not isentropic. 

The initial data is given by,
\begin{align*}
  f(x,y)&=\frac{(x-x_c)^2+(y-y_c)^2}{\R^2},\\
  \Delta u    &= -\beta  \frac{y-y_c}{\R} e^{-f/2},\\
  \Delta v    &=  \beta \frac{x-x_c}{\R}  e^ {-f/2},\\
  \Delta T    &= \frac{1}{2c_p}\beta^2e^{-f},\\
  \rho & = \rho_\infty\left(\frac{T_\infty-\Delta T}{T_\infty}\right)^{1/(\gamma-1)},\\
  u &= (u_\infty+\Delta u),\\
   v &= (v_\infty+\Delta v), \\
    E        &= R\rho\frac{T_\infty-\Delta T}{\gamma-1} +\frac{1}{2}  \rho(u^2+v^2),\\
%
%
%
\end{align*}
where $\R$ scales the size of the vortex and $\beta$ its strength. Furthermore, $(x_c,y_c)$ marks the centre of the vortex, and $u_\infty$ and $v_\infty$ are the velocity components of the freestream. A vortex located at $(x_0,y_0)$ at $t=0$ is advected by the freestream, such that
\begin{align*}
  x_c(t)&=x_0+u_\infty t,\\
  y_c(t)&=y_0+v_\infty t.
\end{align*}
In these computations, the gas constant is, $R=1/\gamma$ and we use the freestream values $\rho_\infty=1$ and $T_\infty=1$ implying that the freestream speed of sound is $c_{\infty}=1$. The direction of the freestream is specified by an angle $\alpha$ against the $x$-axis.  { Hence, $u_\infty=U_\infty\cos\alpha$ and $v_\infty=U_\infty\sin\alpha$ where $U_\infty$ is the freestream speed. } In these numerical experiments, we use
\begin{align*}
\R=0.1,\quad \beta=1.0,\quad \alpha=\pi/4.
\end{align*}
Furthermore, we use $\mu=0.001$. To demonstrate the robustness of the boundary conditions, viscous effects should be significant; otherwise we are essentially testing the Euler case. However, a too large $\mu$ will completely diffuse the vortex before it reaches the boundary. Our choice strikes a balance between the two effects on the computational grids we use. 


The computational domain is the unit square ($[0.1] \times [0,1]$) discretised with $N^2$ equidistant points. In all computations, we use the inviscid fluxes proposed in \cite{Svard22} that were also used in \cite{Svard24_2,Svard25}.

In the first experiment, we initiate the vortex in the centre of the square and use a diagonal freestream ($\alpha=\pi/4$) with $Ma=0.1$, { leading to $U_\infty=0.1$}; the initial data for density is depicted in the top-left panel of Fig. \ref{fig:diag}. We run the scheme with a CFL number of $0.5$. In Fig. \ref{fig:diag}, numerical solutions at $t=15.0$ are shown for three different grids. (With $N=100$ the time step is $\Delta t= 0.0045454$; $N=200$, $\Delta t = 0.002272$; $N=400$, $\Delta t = 0.001136$.)
\begin{figure}[ht]
\centering
\subfigure[$\rho$ at $t=0$, $N=100$]{\includegraphics[width=6cm]{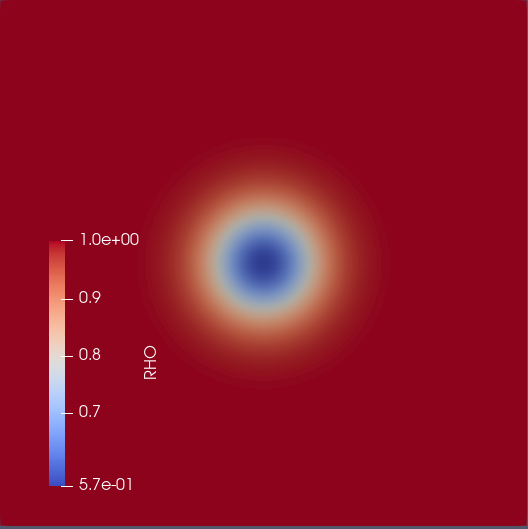}}
\subfigure[$\rho$ at $t=15.0$, $N=100$]{\includegraphics[width=6cm]{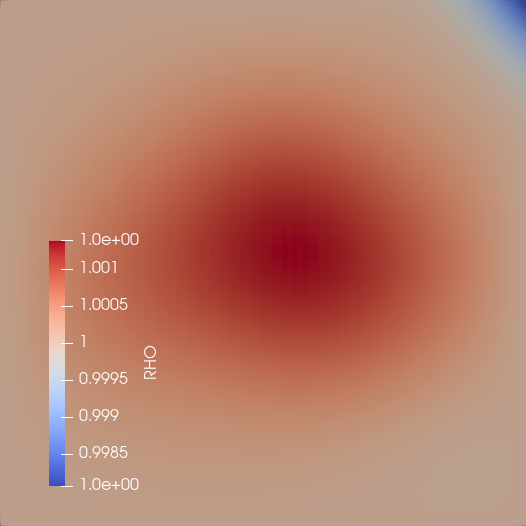}}
\subfigure[$\rho$ at $t=15.0$, $N=200$]{\includegraphics[width=6cm]{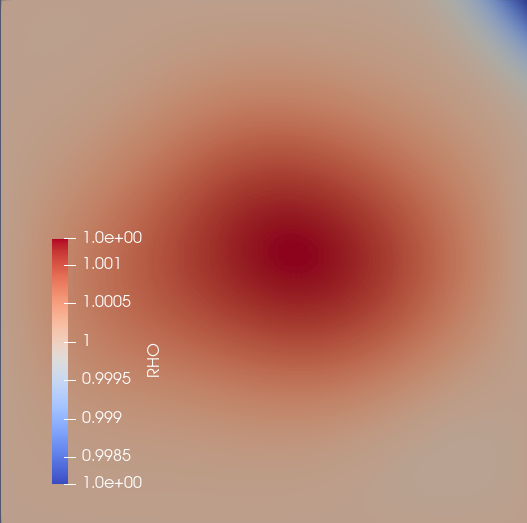}}
\subfigure[$\rho$ at $t=15.0$, $N=400$]{\includegraphics[width=6cm]{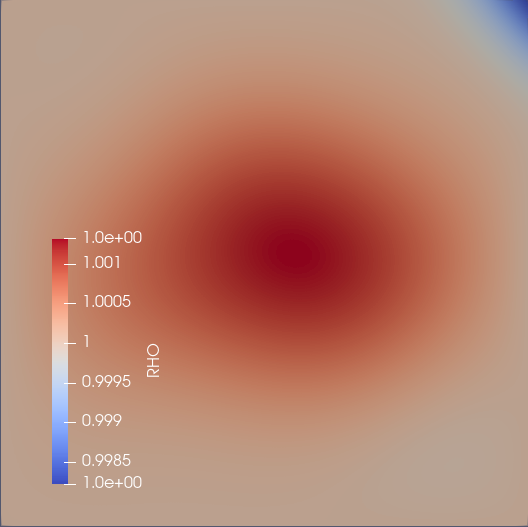}}
\caption{The remains of a vortex that has exited the domain in a diagonal subsonic freestream.}\label{fig:diag}
\end{figure}
Had the domain been infinite, the vortex would be well outside the unit square at $t=15.0$ and the solution in the unit square would have returned to the constant freestream. This should also be the case in the numerical solutions, if the boundary conditions are perfectly transparent. 

However, the boundary conditions are not perfectly transparent.  (They are designed to be entropy-stable.) Nevertheless, since their reflective properties may be of interest, the differences to the freestream are listed in Table \ref{tab:diag_out}. We note that there is little difference between the different grids indicating that the these values represent reflections rather than numerical errors caused by the discretisation scheme.


\begin{table}[htb]
  \caption{Maximal errors for vortex exiting the domain. $t=15.0$ with $Ma=0.1$ and $\alpha=\pi/4$. }\label{tab:diag_out}
  \begin{center}
\begin{tabular}{|c|c|c|c|}
  \hline
  N &  100 & 200 & 400 \\
  \hline
  $\|\rho_{app}-\rho_{ex}\|_\infty$  &0.001370 & 0.001326 & 0.001305 \\
\hline
  $\|E_{app}-E_{ex}\|_\infty$  & 0.003216 & 0.003114  & 0.003066 \\
\hline
\end{tabular}
\end{center}
\end{table}
Initially, the vortex induces $42\%$ perturbation to the freestream density and a $54\%$ perturbation to the freestream energy. The reflections at $t=15.0$ as a percentage of the initial perturbations are $0.3\%$ for both $\rho$ an $E$.

The boundary conditions are of course ``helped'' by the viscosity that diffuses the vortex before it reaches the boundary as well as the subsequent reflections. By successively decreasing $\mu$, the reflections increase somewhat and the results become more and more similar to the inviscid case reported in \cite{Svard25}.

In the second experiment, we turn to the inflow boundary conditions. We place the vortex  well outside the computational domain at $(x,y)=(-0.5,-0.5)$. In a diagonal freestream and  $M=0.1$, the analytical Euler vortex is thus located at the low-left corner when $T=7.1$. By using the vortex solution as boundary data, it should thus enter the computational domain when $T\approx 7$, and at $T=14.2$ it should be roughly centred in the domain. However, since the Euler solution is not an analytical solution to the Navier-Stokes equations, we can not expect a perfect vortex to enter the domain. \footnote{ Since $\mu,\kappa\neq 0$ these data do not represent a perfect vortex. To obtain a perfect vortex, exact analytical expressions of the viscous flux must be added to the inviscid Euler flux data. The point here, however, is merely to demonstrate that the boundary conditions can be used as inflow conditions.} Nevertheless, we have computed this case with $N=200$ and the results are depicted in Fig. \ref{fig:diag_enter}. 
\begin{figure}[ht]
\centering
\subfigure[$\rho$ at $t=7.1$, $N=200$]{\includegraphics[width=6cm]{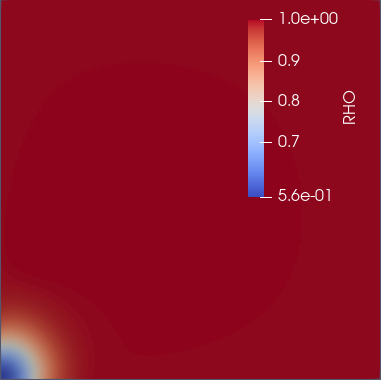}}
\subfigure[$\rho$ at $t=14.2$, $N=200$]{\includegraphics[width=6cm]{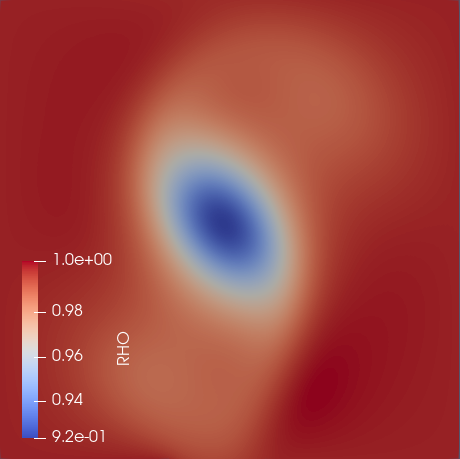}}
\caption{Vortex entering the computational domain.}\label{fig:diag_enter}
\end{figure}
We conclude that the subsonic inflow condition does let the vortex in through the boundary. As, expected, it is distorted and not perfectly symmetric.


\subsection{Circular blast wave}

Next, we compute a circular shock wave emanating from the initial data found in \cite{Svard22}.  It is given by two states at $t=0$: "1" is inside a circle centred at $(x,y)=(0.5,0.5)$ with radius 0.25,  and "2" outside the circle.
\begin{align*}
\rho_1=1,\,\,p_1=1,\\
\rho_2=0.125,\,\,p_2=0.1.
\end{align*}
The velocity is initially zero for both states "1" and "2", and the gas constant is, $R=1/\gamma$. The boundary flux data vector is constructed using the state "2". We stress that the boundary data do not match the shock wave when it arrives at the boundary and hence there will be some reflections.


We run this case on a Cartesian mesh with $N=400$ grid points in each directions. The time step is $\Delta t\approx 0.000417$ and we take $\mu=1.0\cdot 10^{-4}$. The results for $\rho$ are depicted in Fig. \ref{fig:blast}. We note that the boundary treatment is perfectly stable.
\begin{figure}[ht]
\centering
\subfigure[$\rho$ at $t=0.1$, $N=400$]{\includegraphics[width=6cm]{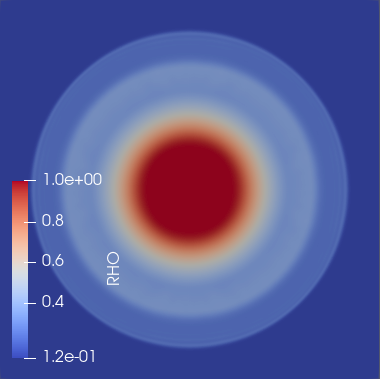}}
\subfigure[$\rho$ at $t=0.2$, $N=400$]{\includegraphics[width=6cm]{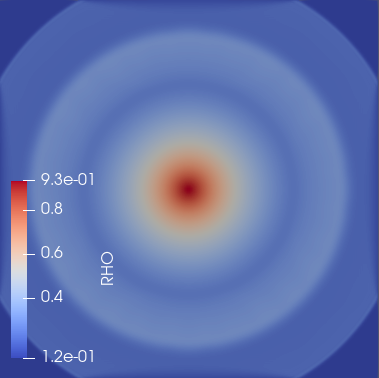}}
\caption{The circular blast wave computed with $\mu=0.0001$.}\label{fig:blast}
\end{figure}

\section{Conclusions}

We have proposed a novel set of entropy-stable boundary conditions for the Navier-Stokes equations.  They can be used as true inflow/outflow conditions, e.g. for channels, and as far-field conditions modelling an open boundary. 

In the inviscid limit, they degenerate to the corresponding Euler boundary conditions proposed in \cite{Svard25}. The new boundary conditions can be shown to be linearly well-posed in all but the subsonic inflow case. We remarked that provable linear well-posedness can be achieved for subsonic inflows by using the same boundary data flux as for the supersonic inflow. However, we did not pursue this option further as it may be detrimental to the regularity of the solution and with $\mu\rightarrow 0$ the Euler boundary conditions in \cite{Svard25} are not recovered.


 To evaluate the robustness of the new boundary conditions, we have run computations of a vortex impinging on an outflow boundary, a vortex entering the domain, and a circular blast wave. All cases run stably. Furthermore, we reported the reflections from the case with a vortex arriving at the boundary, which are less than 0.3\% of the initial perturbations.


\end{document}